\documentclass[12pt,leqno]{amsart}
\usepackage{a4,amssymb,amsthm,amscd,amsmath,verbatim,url,enumerate,mathdots,mathrsfs,cleveref}
\usepackage[pdftex,usenames,dvipsnames]{color}
\usepackage{mathabx}
\usepackage{fancyhdr}
\usepackage{tikz}
\usepackage{tikz-cd}

\title{Four-dimensional quadratic forms over $\cc(\!(t)\!)(X)$}
\author{Parul Gupta}
\address{Universiteit Antwerpen, Departement Wiskunde, Middelheim\-laan~1, 2020 Antwerpen, Belgium.}
\address{IISER Pune, Dr.~Homi Bhabha Road, Pashan, Pune 411 008, India.}
\email{parul.gupta@iiserpune.ac.in}
\thanks{This work was supported by the \emph{Fonds Wetenschappelijk Onderzoek -- Vlaanderen (FWO)} in the \emph{FWO Odysseus Programme} (project `{Explicit Methods in Quadratic Form Theory}'), the \emph{Bijzonder Onderzoeksfonds (BOF), University of Antwerp} (project BOF-DOCPRO-4, 2865), and the \emph{Science and Engeneenring Research Board  (SERB)}, India
(Grant CRG/2019/000271).}
\date{\today}

\newcommand{\la}{\langle}
\newcommand{\ra}{\rangle}

\newcommand{\qq}{\mathbb Q}

\newcommand{\cc}{\mathbb C}

\newcommand{\nat}{\mathbb{N}} 
\newcommand{\zz}{\mathbb Z}

\newcommand{\mc}[1]{\mathcal{#1}}
\newcommand{\mf}[1]{\mathfrak{#1}}

\newcommand{\mg}[1]{{#1}^{\times}}
\newcommand{\sq}[1]{{#1}^{\times 2}}

\newcommand{\ovl}{\overline}

\swapnumbers
\numberwithin{equation}{section}

\newtheorem*{thm*}{Theorem}

\newtheorem*{cor*}{Corollary}
\newtheorem*{cor1}{Corollary 1}
\newtheorem*{cor2}{Corollary 2}

\newtheorem*{lem*}{Lemma}
\newtheorem*{qu*}{Question}

\theoremstyle{definition}

\renewenvironment{proof}{\par\noindent {\em Proof:}}{\hfill$\Box$\medskip}
\theoremstyle{plain}

\begin{document}

\begin{abstract}
For quadratic forms in $4$ variables defined over the rational function field in one variable over $\cc(\!(t)\!)$, the validity of the local-global principle for isotropy  with respect to different sets of discrete valuations is examined.  

\medskip
\noindent
{\sc{Classification (MSC 2010):} 11E04,12E30, 12J10} 

\end{abstract}

\maketitle

\section{Introduction}

Let $E$ be a field of characteristic different from $2$ and let $E(X)$ denote the rational function field in one variable over $E$.  

For  $ E = \cc(\!(t)\!)$, the field of Laurent series in one variable over the complex numbers,
the quadratic form 
$$Y_1^2+tY_2^2+tY_3^2+X(Y_1^2+Y_2^2+ t Y_4^2)$$
in the variables $Y_1, Y_2,Y_3,Y_4$ over $E(X)$ has no non-trivial zero, but it  has a non-trivial zero over the completion of $E(X)$ with respect to any non-trivial valuation on $E(X)$ that is trivial on $E$.
This is in contrast to the situation  when $E$ is a finite field, by the Hasse-Minkowski Theorem (See \cite[Chapter VI, Theorem 66.1]{OM}). 
Note that, in both cases, the field $E$ has a unique extension of each degree in a fixed algebraic closure.

By a \emph{$\zz$-valuation}, we mean a valuation with value group $\zz$. 
A quadratic form is \emph{isotropic} if it has a non-trivial zero, otherwise it is \emph{anisotro\-pic}.
In all generality,
an anisotropic quadratic form over $E(X)$ of dimension at most $3$ remains anisotropic over the completion of $E(X)$ with respect to some $\zz$-valuation on $E(X)$ that is trivial on $E$; this follows for example from Milnor's Exact Sequence \cite[Theorem IX.3.1]{Lam}. 
The case of $4$-dimensional quadratic forms  is the first case over $E(X)$  where the validity of such a local-global principle for isotropy  depends on the base field $E$.

When $E$ is a nondyadic local field, using a result of Lichtenbaum \cite{Lichtenbaum}, one obtains that a $4$-dimensional anisotropic quadratic form over $E(X)$ remains anisotropic over the completion of $E(X)$ with respect to some $\zz$-valuation on $E(X)$ that is trivial on $E$  (see \cite[Remark 3.8]{CTPS}).
This resembles the case where $E$ is a finite field.

In contrast to the situations where $E$ is a finite field or a local field, for $E=\cc(\!(t)\!)$ the example of the quadratic form above shows that the local-global principle for isotropy of $4$-dimensional quadratic forms over $E(X)$ fails with respect to $\zz$-valuations that are trivial on $E$. 
However, anisotropy of this quadratic form can be detected over the larger field $\cc(X)(\!(t)\!)$, by using Springer's Theorem (see \cite[Proposition VI.1.9]{Lam}). 

Consider the more general situation where the field $E$ is complete with respect to a nondyadic $\zz$-valuation $v$. 
In this case, a local-global principle for isotropy was obtained in \cite{CTPS} using a geometric setup. 
Let $\mc O_v$ denote the valuation ring of $v$.
By a  \emph{model for $E(X)$ over $\mc O_v$} we mean a two-dimensional integral normal projective flat $\mc O_v$-scheme $\mathscr X$ whose function field is isomorphic to $E(X)$.
Codimension-one points on a model of $E(X)$ over $\mc O_v$ correspond to certain $\zz$-valuations on $E(X)$.
For a model $\mathscr X$ of $E(X)$ over $\mc O_v$ let $\Omega_\mathscr X$ denote the set of $\zz$-valuations given by codimension-one points of $\mathscr X$.  
Consider the set $\Omega= \bigcup_{\mathscr X } \Omega_\mathscr X$ where the union is taken over all models $\mathscr X$ of $E(X)$ over $\mc O_v$.
It follows from  \cite[Theorem 3.1 and Remark 3.2]{CTPS} that an anisotropic quadratic form over $E(X)$ remains anisotropic over the completion of $E(X)$  with respect to some $\zz$-valuation in $\Omega$.
One may ask whether this remains true if one replaces $\Omega$ by $\Omega_{\mathscr X}$ for some well-chosen model $\mathscr X$ of $E(X)$ over $\mc O_v$.

The aim of this note is to show that this is not the case: 
 if the residue field of $v$ is separably closed then, for any model $\mathscr X$ of $E(X)$ over $\mc O_v$,  there exists an anisotropic  $4$-dimensional quadratic form over $E(X)$ which is isotropic over the completion of $E(X)$ with respect to any $w\in \Omega_{\mathscr X}$   (Corollary 2).
Let $\pi \in \mc O_v$ be a uniformiser of~$v$.
For any model $\mathscr X$ of $E(X)$ over $\mc O_v$, the set $\{w(\pi) \mid  w \in\Omega_{\mathscr X} \}$ is finite and hence  it has an upper bound.   
However, for any positive integer $r$, the quadratic form  
$$\varphi_r = (X^r-\pi)Y_1^2+ (X^{r+1}+\pi)Y_2^2 + \pi X Y_3^2+  X(X^r+\pi)Y_4^2 $$
is anisotropic over $E(X)$,  but it is isotropic over the completion of $E(X)$ with respect to any $\zz$-valuation $w$ on $E(X)$ with $w(\pi) <r$ (Theorem).  
The construction of $\varphi_r$ is inspired by the example in \cite[Remark 3.6]{CTPS} of an anisotropic $6$-dimensional quadratic form over $\qq_p(X)$ where $p$ is an odd prime.

\section{Results}

We assume some familiarity with basic quadratic form theory over fields, for which we refer to \cite{Lam}.    
We first fix some notation and recall some results.

By a \emph{quadratic form} or simply a \emph{form} we mean a regular quadratic form. 
Let $E$ always be a field of characteristic different from $2$ and let $\mg{E}$ denote its  multiplicative group.
For $a_1, \ldots, a_n \in \mg E$ the diagonal form $a_1X_1^2+\dots+a_nX_n^2$ is denoted by $\la a_1, \ldots,a_n\ra$.

Let $v$ be a $\zz$-valuation on $E$.
We denote the corresponding
valuation ring, its maximal ideal and its residue field respectively  by $\mc O_v, \mf m_v$ and $\kappa_v$.
For an element $a\in \mc O_v$, let $\ovl{a}$ denote the image $a+\mf m_v$ of $a$ under the residue map $\mc O_v \rightarrow \kappa_v$.
The completion of $E$ with respect to $v$ is denoted by $E_v$. 
We say that $v$ is \emph{henselian} if it extends uniquely to every finite field extension of $E$.
Complete discretely valued fields are henselian (see \cite[Theorem 1.3.1 and Theorem 4.1.3]{EP}).
We recall a consequence of Hensel's Lemma:  
\begin{lem*}
Let $v$ be a henselian $\zz$-valuation on $E$ such that $v(2) =0$. Then
\begin{enumerate}[$(a)$]
\item The form $\la u_1, u_2 \ra$ over $E$ is isotropic if and only if $\ovl{u_1u_2} \in -\sq{\kappa}_v$. 
\item
If $\kappa_v$ is separably closed, then every $3$-dimensional form over $E$ is isotropic. 
\end{enumerate}
 \end{lem*}
\begin{proof}
Since  $\ovl{u_1u_2} \in -\sq{\kappa}_v$ the polynomial equation $t^2 + \ovl{u_1u_2}$ has a solution in $\kappa_v$ and  since $v(2) =0$ it follows by Hensel's Lemma \cite[Theorem 4.1.3$(4)$]{EP} that $u_1u_2 \in -E^2$, whereby the quadratic form $\la u_1, u_2 \ra$ over $E$ is isotropic.
Since $\kappa_v$ is separably closed with $v(2) =0$, we have that $\ovl{u} \in -\sq{\kappa}_v$ for all $u \in \mg {\mc O}_v$.
Since every $3$-dimensional quadratic form over $E$ contains a $2$-dimensional form isometric to $\lambda \la 1 ,u\ra$ for some $u \in \mg {\mc O}_v$ and $\lambda \in \mg E$;  $(b)$ follows from $(a)$. 
\end{proof}

The set of all $\zz$-valuations on $E(X)$ is denoted by $\Omega_{E(X)}$.
For $r\in \nat$, we define
$$\Omega_r = \{ w\in \Omega_{E(X)} \mid  w(\mg E) =i\zz \mbox{ for some } 0\leq i \leq r \}.$$  
With this notation, $\Omega_0$ is the set of all $E$-trivial $\zz$-valuations on $E(X)$.
We recall that any monic irreducible polynomial $p \in E[X]$ determines a unique $\zz$-valuation $v_p$ on $E(X)$ which is trivial on $E$ and such that $v_p(p)=1$.
There is further a unique $\zz$-valuation $v_\infty$ on $E(X)$ such that $v_\infty(f)=-\deg(f)$ for any $f\in E[X]\setminus\{0\}$.
Moreover, every $\zz$-valuation $w$ on $E(X)$ trivial on $E$ is either equal to $v_\infty$ or to $v_p$ for some monic irreducible polynomial $p \in E[X]$ (see \cite[Theorem 2.1.4]{EP}), and in either of the two cases the residue field is a finite field extension of $E$.

\begin{thm*}
Let $v$ be a henselian $\zz$-valuation on $E$ such that $v(2) =0$.
Assume that $\kappa_v$ is separably closed. 
Let $\pi\in \mg E$ be such that $v(\pi) =1$ and let $r \in \nat$. 
Then the quadratic form  
$$\varphi_r = \langle X^r-\pi, X^{r+1}+\pi, \pi X, X(X^r+\pi) \rangle$$
 is isotropic over $E(X)_w$ for every $\zz$-valuation $w \in \Omega_{r-1}$ but anisotropic over $E(X)_w$ for some $w\in \Omega_r$.
 
\end{thm*}

\begin{proof}
Set $F=E(X)$.
We first show that $\varphi_r$ is isotropic over $F_w$ for all $w\in \Omega_{r-1}$. 
Consider $w\in \Omega_{r-1}$.

\smallskip
\noindent \underline {Case 1:} $w(\pi) =0 = w(X)$.
Then $\kappa_w$ is a finite extension of $E$. Since $v$ is henselian, there is a unique extension $v'$ of $v$ to $\kappa_w$, and $v'$ again henselian. Furthermore, it follows by \cite[Theorem 3.3.4]{EP} that $v'(\mg {\kappa}_w)$ is isomorphic to $\zz$ and $\kappa_{v'}$ is separably closed. It follows by part $(b)$ of the Lemma that every $3$-dimensional quadratic form over $\kappa_w$ is isotropic.
We have that $w =v_p$ for some monic irreducible polynomial $p \in E[X]$ such that $p \neq X$.  
Note that, in this case at least three diagonal coefficients of $\varphi_r$ are units in $\mc O_w$.
It follows by Springer's Theorem \cite[Proposition VI.1.9]{Lam} that $\varphi_r$ is isotropic over $F_w$.

\smallskip    
 
\noindent \underline {Case 2:} $0 \leq w (\pi)< r$ and $1\leq w(X)$. 
Let $ u = (X^r\pi ^{-1}-1)(X^{(r+1)}\pi^{-1}+1)$.
Then $w(u) =0$ and $\ovl{u} =-1 \in - \sq{\kappa}_w$.
It follows by part $(a)$ of the Lemma that the form $\pi^{-1} \la  X^r-\pi,  X^{r+1}+\pi \ra $ is isotropic over $F_w$.
Thus $\varphi_r$ is isotropic over ${F}_w$.

\smallskip
\noindent \underline {Case 3:} $w(X) < 0 \leq  w (\pi)< r$. 
Note that $\kappa_w$ is either a finite extension of $E$ or a rational function field over a finite extension of $\kappa_v$; since $-1\in \sq{\kappa}_v$, we get in either case that $-1 \in \sq{\kappa}_w$.
Consider $ u = (1+  \pi X^{-{(r+1)}})(1+\pi X^{-r})$. 
We have that $w(u) =0$ and $\ovl{ u } = 1 \in \sq{\kappa}_w = -\sq{\kappa}_w$. 
It follows by part $(a)$ of the Lemma that the form $X^{-(r+1)} \la X^{r+1}+\pi , X(X^{r}+\pi)\ra$ is isotropic over $F_w$.
Thus $\varphi_r$ is isotropic over $F_w$.

\smallskip

We have thus shown that $\varphi_r$ is isotropic over $F_w$ for every $w \in \Omega_{r-1}$. 
Now we show that $\varphi_r$ is anisotropic over $F_w$ for some $w\in \Omega_F$.

Let $E' = E(s)$, where $s =\sqrt[r]{\pi}$. Then $v$ extends uniquely to a  valuation on $E'$ which we again denote by $v$. Note that $s^r = \pi$ in $E'$ and hence $v (\pi) = rv(s)$. Then $v' = rv$ is a $\zz$-valuation on $E'$. 

Let $L= E'(X)$ and let $Y =\frac{X}{s}$. 
Note that $L =E'(Y)$.
By \cite[Corollary 2.2.2]{EP}, there exists a unique extension of $v'$ to $L$ such that $v(Y) =0$ and $\ovl{Y}$ is transcendental of $\kappa_{v'}$; we further have that $\kappa_w = \kappa_{v'}(\ovl{Y})$ and $w(\mg L) = v'(\mg {E'}) = \zz$.
Since $w(Y) =0$, we have that $w(X) =w(s) =1$. 
We get that
$$\varphi_r = \langle s^r(Y^r-1), s^r(sY+1), s^{r+1}Y, s^{r+1}Y(Y^r+1) \rangle $$
Consider the forms $\varphi_1 = \langle Y^r-1,~ sY+1\rangle$ and $\varphi_2 = \langle Y,~ Y(Y^r+1) \rangle$. 

Since $ \ovl{Y}^r-1, \ovl{Y}^{r} +1 \notin - \sq{\kappa}_w$, it follows by Springer's Theorem \cite[Proposition VI.1.9]{Lam} that the quadratic form $s^{-r}\varphi_r$ is anisotropic over~$L_w$. 
Hence $\varphi_r$ is anisotropic over $L_w$.
We obtain that $\varphi_r$ is anisotropic over $F_{w|_F}$. 
Note that, $w(\pi) = w(s^r) =rw(s) =r$, thus $w\in \Omega_r$.
\end{proof}

We now provide a different perspective to the above theorem.
For a subset $\Omega\subseteq \Omega_{E(X)}$,  
we say that $\Omega$ has the \emph{finite support property} if for every $f \in \mg{E(X)}$ the set $\{w\in \Omega\mid w(f) \neq 0\}$ is finite.
It is well-known that $\Omega_{0}$ has the finite support property.
When $E$ carries a discrete valuation the set $\Omega_{E(X)}$ does not have the finite support property.
However, for any model $\mathscr X$ of $E(X)$ over $\mc O_v$, the set $\Omega_\mathscr X$ contains $\Omega_0$ and  has the finite support property. 
We show the following:

\begin{cor1}
Let $v$ be a henselian $\zz$-valuation on $E$ with $v(2) =0$. 
Assume that $\kappa_v$ is separably closed. Let $\Omega \subseteq \Omega_{E(X)}$ be a subset with the finite support property. 
Then there exists an anisotropic  $4$-dimensional quadratic form over $E(X)$ which is isotropic over $E(X)_w$ for every $w\in \Omega$.
\end{cor1}

\begin{proof}
Let $\pi \in \mg E$ be such that $v(\pi) =1$.
Since $\Omega$ has the finite support property, the set $\{w\in \Omega\mid w(\pi) \neq 0  \}$
is finite. 
Set 
$r= 1+\max\{w(\pi) \mid w\in \Omega\}$.
Clearly $\Omega\subseteq\Omega_{r-1} $.
Then the form $\varphi_r$ in the Theorem is isotropic over $E(X)_w$ for every $w\in \Omega$, but anisotropic over $E(X)$. 
\end{proof}

\begin{cor2}
Let $v$ be a henselian $\zz$-valuation on $E$ with $v(2) =0$. Assume that $\kappa_v$ is separably closed.
Let $\mathscr X$ be a regular model of $E(X)$ over $\mc O_v$.
Then there exists an anisotropic  $4$-dimensional quadratic form over $E(X)$ which is isotropic over $E(X)_w$ for every $w\in \Omega_{\mathscr X}$.
\end{cor2}

\begin{proof}
By  \cite[Chapter ~II, Lemma~6.1]{H}, for every element $f\in \mg {E(X)}$ the set $\{w \in\Omega_{\mathscr X} \mid w(f) \neq 0 \}$ is finite, hence the statement follows by Corollary 1.
\end{proof}

\subsubsection*{Acknowledgments}
The author wishes to thank David Leep, Suresh Venapally, Karim Johannes Becher, Gonzalo Manzano Flores and Marco Zaninelli for many inspiring discussions and comments related to this article. 
This article is based on the author's PhD-thesis, prepared under the supervision of Karim Johannes Becher (\emph{Universiteit Antwerpen}) and Arno Fehm (\emph{Technische Universit\"at Dresden}) in the framework of a joint PhD at \emph{Universiteit Antwerpen} and \emph{Universit\"at Konstanz}.

\bibliographystyle{amsalpha}

\end{document}